\documentclass[reqno]{amsart}

\usepackage{amsmath}
\usepackage{amsfonts}
\usepackage{amssymb,enumerate}
\usepackage{amsthm}
\usepackage[all]{xy}
\usepackage{rotating}
\usepackage{hyperref}

\theoremstyle{plain}
\newtheorem{lem}{Lemma}[section]
\newtheorem{cor}[lem]{Corollary}
\newtheorem{prop}[lem]{Proposition}
\newtheorem{thm}[lem]{Theorem}

\theoremstyle{definition}

\newtheorem{construction}[lem]{Construction}

\newtheorem{para}[lem]{}

\newtheorem*{convention*}{Convention}

\newcommand{\pdim}{\operatorname{pd}}	
\newcommand{\pd}{\operatorname{pd}}

\newcommand{\id}{\operatorname{id}}	
\newcommand{\fd}{\operatorname{fd}}

\newcommand{\cidim}{\mathrm{CI}\text{-}\!\dim}

\newcommand{\DD}{\operatorname{D}}

\newcommand{\HH}{\operatorname{H}}

\newcommand{\spec}{\operatorname{Spec}}

\newcommand{\aq}{\operatorname{AQ-dim}}

\newcommand{\Ker}{\operatorname{Ker}}

\newcommand{\ideal}[1]{\mathfrak{#1}}
\newcommand{\m}{\ideal{m}}
\newcommand{\n}{\ideal{n}}
\newcommand{\p}{\ideal{p}}
\newcommand{\q}{\ideal{q}}
\newcommand{\fm}{\ideal{m}}

\newcommand{\fa}{\ideal{a}}

\newcommand{\fr}{\ideal{r}}

\newcommand{\bbn}{\mathbb{N}}

\newcommand{\from}{\leftarrow}
\newcommand{\xra}{\xrightarrow}
\newcommand{\xla}{\xleftarrow}

\newcommand{\vf}{\varphi}

\renewcommand{\geq}{\geqslant}
\renewcommand{\leq}{\leqslant}
\renewcommand{\ker}{\Ker}

\newcommand{\Otimes}[3][R]{#2\otimes_{#1}#3}

\newcommand{\cfd}{\operatorname{CI-fd}}
\newcommand{\cfdd}{\operatorname{CI*-fd}}
\newcommand{\cdim}{\operatorname{CI\text{-}dim}}
\newcommand{\ciddim}{\operatorname{CI*\text{-}dim}}
\newcommand{\cid}{\operatorname{cidim}}
\newcommand{\cidd}{\operatorname{cidim-fd}}

\numberwithin{equation}{lem}

\begin{document}

\bibliographystyle{amsplain}

\title[Andr\'e-Quillen homology \& complete intersection flat dimensions]{Andr\'e-Quillen homology and ascent of complete intersection flat dimensions}

\author{Keri Ann Sather-Wagstaff}
\address{K.~A.~Sather-Wagstaff,
Framingham State University, 100 State Street
PO Box 9101
Framingham, MA 01701, USA}
\email{ksatherwagstaff@framingham.edu}

\author{Tirdad Sharif}
\address{T.~Sharif, School of Mathematics,
Institute for Research in Fundamental Sciences (IPM),
P.~O.~Box 19395-5746, Tehran, Iran.}
\email{sharif@ipm.ir}


\begin{abstract}
Using Andr\'{e}-Quillen homology, we prove an ascent result for different types of complete intersection flat dimensions along an essentially of finite type flat local homomorphism with complete intersection closed fiber. As an application of our result, we extend a result of Majadas Soto and of Avramov, Henriques and \c{S}ega on the Andr\'{e}-Quillen dimension of surjective local homomorphisms to that of essentially of finite types.
\end{abstract}

\maketitle

\section{Introduction}
\label{introduction}

In this paper, all rings are commutative and Noetherian; $(R, \fm, k)$, $(S,\n, \ell)$, and $(T,\fr, v)$ are local rings; and
$\varphi\colon R\to S$ and $\sigma\colon S\to T$ are local ring homomorphisms.

Following Avramov and Iyengar~\cite{AI2}, we investigate connections between the vanishing of Andr\'{e}-Quillen homology and the structure of a local ring homomorphism, focusing on complete intersection flat dimensions;
see Section~\ref{background} for background.
In this direction, we use the Andr\'{e}-Quillen dimension~\eqref{eq250805a} and a result of Avramov, Foxby, and Herzog~\cite[Lemma (3.2)]{AFH1} to
establish an ascent result in Theorem~\ref{thm3.1} for complete intersection flat dimensions along an essentially of finite type flat local homomorphism with complete intersection closed fiber.
%
%
We push this further to the case where $\vf$ has finite flat dimension in Corollary~\ref{cor250803a}.

In~\cite[Theorem B]{Sh}, Sharif extended a result of Majadas Soto~\cite[Proposition~12]{S}---obtained separately by Avramov,
Henriques and \c{S}ega~\cite[(2.5)(3)]{AHS} with a different proof---on Andr\'{e}-Quillen homology from the case where $\varphi$ is surjective
to the case where $\varphi$ is essentially of finite type, and $S$ is a complete local ring. As an application of Theorem~\ref{thm3.1}(i), we prove~\cite[Theorem~B]{Sh} in Theorem~\ref{thm3.2} without the completeness assumption on $S$, and with a much shorter proof. 
%
%
Lastly, using Theorem~\ref{thm3.1}(ii), we extend~\cite[Theorem A (iv)$\implies$(v)]{Sh}
to the case where $\sigma$ is essentially of finite type; see Corollary~\ref{cor3.4}.

\section{Some Background Material}
\label{background}

Central to this work are the following Cohen factorizations of
Avramov, Foxby, and B. Herzog~\cite{AFH1}.

\begin{para}\label{para1}
Let $\widehat{R}$ and $\widehat{S}$ be the $\m$-adic and $\n$-adic completions of $R$ and $S$, respectively.
The \emph{completion} of $\varphi$ is the induced local ring homomorphism $\widehat{\varphi}\colon \widehat{R} \to \widehat{S}$. The \emph{semi-completion} of $\varphi$ is the local homomorphism
$\grave{\varphi}\colon R \to \widehat{S}$, which is obtained by composing the canonical inclusion $R\to \widehat{R}$ with $\widehat{\varphi}$, or equivalently,
by composing $\varphi$ with the canonical inclusion $S \to \widehat{S}$.
A \emph{regular  factorization of $\vf$},
is a diagram of local ring homomorphisms
$R\xra{\dot{\varphi}} R'\xra{\varphi'} S$
such that $\varphi=\varphi'\dot{\varphi}$, the map $\dot{\varphi}$ is \emph{weakly regular}
(i.e., flat with regular \emph{closed fibre} $R'/\fm R'$), and $\varphi'$ is surjective. A \emph{Cohen factorization of $\vf$} is a regular factorization
$R\xra{\dot{\varphi}} R'\xra{\varphi'} S$ of $\vf$ such that $R'$ is complete.
The existence of a Cohen factorization for $\grave{\varphi}$, is
guaranteed by~\cite[Theorem (1.1)]{AFH1}.
\end{para}

\begin{para}\label{para2}
The local ring homomorphism $\varphi\colon R\to(S,\n)$ is called \emph{complete intersection} or \emph{c.i.}
at $\n$, if in a given Cohen factorization $R\to R'\xra{\varphi'} \widehat{S}$ of
$\grave{\varphi}$, the ideal $\ker\varphi'$ is generated by an $R'$-regular sequence.
The c.i.~property for $\varphi$ is independent of the choice of Cohen factorization
by~\cite[Remark (3.3)]{Av}.
For instance, if $\vf$ is surjective, then it is c.i.\ at $\n$ if and only if $\Ker (\vf)$ is generated by an $R$-regular sequence. And if $\vf$ is flat, then by~\cite[Lemma (5.2.1)]{Av} it is c.i.\ at $\n$ if and only if the closed fibre is a complete intersection, i.e., $\widehat{S/\m S}\cong Q/\fa$ where $Q$ is regular and $\fa$ is generated by a $Q$-regular sequence.
Also $\varphi$ is called \emph{locally complete
intersection} or \emph{l.c.i}, if for each $\q\in \spec(S)$, the local homomorphism $\varphi_\q\colon R_{\p}\rightarrow S_{\q}$ is c.i.~at $\q S_{\q}$ where $\p=\vf^{-1}(\q)$.
\end{para}

\begin{para}\label{para3}
If $\widehat{R}$ is the $\m$-adic completion of $R$, then Andr\'{e}-Quillen homology
satisfies $\DD_n(k|R,k)=\DD_n(k|\widehat{R},k)$ for $n\geq 0$, by~\cite[(10.18)]{An2}.
Using the Jacobi-Zariski exact sequence arising from $R\to \widehat{R}\to k$, one concludes that $\DD_n(\widehat{R}|R,k)=0$ for $n\geq 0$.
If we take the map between the Jacobi-Zariski exact sequences arising from the diagrams $\varphi\colon R\to S \to \ell$ and $\widehat{\varphi}\colon \widehat{R}\to \widehat{S}\to \ell$, then it follows readily that $\DD_n(\widehat{S}|\widehat{R},\ell)=\DD_n(S|R,\ell)$ for $n\geq 0$.
\end{para}

\begin{para}\label{defn130610a}
Let $R\xra{\dot{\varphi}} R'\xra{\varphi'} \widehat{S}$ be a Cohen factorization of $\grave{\varphi}$, and set $J'=\Ker\varphi'$. The \emph{$n$th Koszul homology} of $\varphi$ is $\HH_n(K(\varphi)):=\HH_n(K^{R'}(J'))$
where $K^{R'}(J')$ is the Koszul complex associated to a minimal set of generators for ideal $J'$.
The module $\HH_n(K(\varphi))$ is independent of Cohen factorization of $\grave{\varphi}$; see~\cite[Fact (2.9)]{Sh}.
\end{para}

\begin{para}\label{fact130625a}
Fix a finitely generated $R$-module $M$ and an arbitrary $R$-module $N$.

Avramov, Gasharov and Peeva~\cite{AGP} introduced and studied the complete intersection dimension of $M$:
\begin{align*}
\cidim_R M
&:=\inf\left\{\fd_Q(R'\otimes_R M)-\pd_Q(R')\left| \text{
\begin{tabular}{@{}c@{}}
$R\to R'\from Q$ is a \\ quasi-deformation
\end{tabular}
}\!\!\!\right. \right\}
\end{align*}
where a \emph{quasi-deformation} $R\xra{f} R'\xla{g} Q$ is a diagram of local ring homomorphisms such that $f$ is flat, and $g$ is surjective and c.i.\ at the maximal ideal of $R'$.
In~\cite{SSW1} Sather-Wagstaff introduced and studied the \emph{complete intersection flat dimension} of $N$,
as the natural extension of $\cidim_R M$:
\begin{align*}
\cfd_R N
&:=\inf\left\{\fd_Q(R'\otimes_R M)-\pd_Q(R')\left| \text{
\begin{tabular}{@{}c@{}}
$R\to R'\from Q$ is a \\ quasi-deformation
\end{tabular}
}\!\!\!\right. \right\}.
\end{align*}
The \emph{upper complete intersection dimension} and  \emph{weakly complete intersection dimension} of $M$
were introduced by Takahashi~\cite{RT} and Majadas Soto~\cite{M1}, respectively:
\begin{align*}
\ciddim_R M
&:=\inf\left\{\fd_Q(R'\otimes_R M)-\pd_Q(R')\left| \text{
\begin{tabular}{@{}c@{}}
$R\to R'\from Q$ is an \\ upper quasi-deformation
\end{tabular}
}\!\!\!\right. \right\}\\
\cid_R M
&:=\inf\left\{\fd_Q(R'\otimes_R M)-\pd_Q(R')\left| \text{
\begin{tabular}{@{}c@{}}
$R\to R'\from Q$ is \\ complete intersection\\ quasi-deformation
\end{tabular}
}\!\!\!\right. \right\}
\end{align*}
where an \emph{upper quasi-deformation} is a quasi-deformation $R\xra f R'\from Q$
in which the flat map $f$ is weakly regular, and
a \emph{complete intersection quasi-deformation}
is a quasi-deformation $R\xra f R'\from Q$
in which the flat map $f$ is c.i.\ at the maximal ideal of $R'$, i.e., $f$ is flat with c.i.\ closed fibre by~\ref{para2}.
We complement these with the \emph{upper complete intersection flat dimension} and the \emph{weakly complete intersection flat dimension} of $N$:
\begin{align*}
\cfdd_R N
&:=\inf\left\{\fd_Q(R'\otimes_R N)-\pd_Q(R')\left| \text{
\begin{tabular}{@{}c@{}}
$R\to R'\from Q$ is an \\ upper quasi-deformation
\end{tabular}
}\!\!\!\right. \right\}\\
\cidd_R N
&:=\inf\left\{\fd_Q(R'\otimes_R N)-\pd_Q(R')\left| \text{
\begin{tabular}{@{}c@{}}
$R\to R'\from Q$ is \\ complete intersection\\ quasi-deformation
\end{tabular}
}\!\!\!\right. \right\}.
\end{align*}
From the definitions, we see that $\cfdd_R N\geq\cidd_R N\geq\cfd_R N$.
\end{para}

In the proof of Corollary~\ref{cor3.4}, we use the following construction of Iyengar and
Sather-Wagstaff~\cite[(5.9)]{IW}.

\begin{construction}\label{fact130607c}
Let $R\xra{\dot{\vf}}R'\xra{\vf'} S$ and $R'\xra{\dot{\rho}}
R''\xra{\rho'} T$
be Cohen factorizations of $\vf$ and $\sigma\vf'$,
respectively. Then $\rho'$ factors through the tensor
product $S'=R''\otimes_{R'}S$ and gives the following commutative
diagram of local ring homomorphisms
\[
\xymatrixrowsep{2.5pc}
\xymatrixcolsep{2.5pc}
\xymatrix{
&& R''
\ar@{->}[dr]^{\vf''}
\ar@/^2.5pc/[ddrr]^{\rho'=\sigma'\vf''}
\\
& R'
\ar@{->}[dr]^{\vf'}
\ar@{->}[ur]^{\dot\rho}
&& S'
\ar@{->}[dr]^{\sigma'}
\\
R
\ar@{->}[rr]^{\vf}
\ar@/^2.5pc/[uurr]^{\dot\rho \dot\vf}
\ar@{->}[ur]^{\dot\vf}
&& S
\ar@{->}[rr]^{\sigma}
\ar@{->}[ur]^{\dot\sigma}
&& T
}
\]
in which $\dot{\sigma}$ and $\vf''$ are the natural maps to
the tensor product and the diagrams
$S\to S'\to T$,
$R\to R''\to T$, and
$R'\to R''\to S'$ are Cohen factorizations.
\end{construction}

\section{Main Results}\label{results}

Our first main result is an ascent result for complete intersection flat dimensions.

\begin{thm}\label{thm3.1}
Let $M$ be a finitely generated $S$-module. If $\varphi$ is flat, essentially of finite type, and c.i.\ at $\n$, then the following hold.

\begin{itemize}
\item[(i)]
If $\cfd_R M<\infty$, then $\cidim_S M<\infty$.
\item[(ii)]
If $\cidd_R M<\infty$, then $\cid_S M<\infty$.
\item[(iii)]
Let $\ell$ be separable over $k$. If $\cfdd_R M<\infty$, then $\ciddim_S M<\infty$.
\end{itemize}
\end{thm}

\begin{proof}
(i) Since $\cfd_R M<\infty$, there is a quasi-deformation
\begin{equation}
R\xra{\delta} (R',\m',k')\xla{\gamma} Q\label{eq250803d}
\end{equation}
such that $\fd_Q(R'\otimes_{R} M)<\infty$.
Set $T=R'\otimes_{R} S$. Since $\varphi$ is essentially of finite type, $T$ is Noetherian, and there is a maximal
ideal $\mathfrak M$ of $T$ that contracts to $\m'$ in $R'$ and to $\n$ in $S$. Consider the completion $(\widehat{T_{\mathfrak M}},\mathfrak M\widehat{T_{\mathfrak M}},E)$ of the local ring $T_{\mathfrak M}$.
The natural maps $\alpha\colon  S\rightarrow \widehat{T_{\mathfrak M}}$ and $\beta\colon  R'\rightarrow \widehat{T_{\mathfrak M}}$ are flat and local.

We claim that $\DD_2(\widehat{T_{\mathfrak M}}|R',E)=0$.
By our assumptions, $\varphi$ is a flat local homomorphism with c.i.\ closed fiber. Thus from~\ref{para2} and ~\cite[Proposition (1.8)]{Av} we get $\DD_2(S|R,\ell)=0$.
By~\cite[(3.20)]{An2}, we have $\DD_2(S|R,E)\cong\DD_2(S|R,\ell)\otimes_{\ell} E=0$.
Using flat base change and localization~\cite[(4.54)]{An2} and~\cite[(4.59), (5,27)]{An2}, our claim now follows from~\ref{para3}. 

Thus~\cite[Proposition (1.8)]{Av} implies that $\beta$ is c.i.\ at $\mathfrak M\widehat{T_{\mathfrak M}}$. Since $\gamma$ is a deformation, we have $\DD_2(R'|Q,k')=0$, and again using~\cite[(3.20)]{An2}, we get $\DD_2(R'|Q,E)=0$. The Jacobi-Zariski exact sequence arising from the diagram $Q\xra{\gamma} R'\xra{\beta} \widehat{T_{\mathfrak M}}$ and the above vanishing results imply $\DD_2(\widehat{T_{\mathfrak M}}|Q,E)=0$. Thus, the local homomorphism $\theta=\beta\gamma$ is c.i.\ at
$\mathfrak M\widehat{T_{\mathfrak M}}$ by~\cite[Proposition (1.8)]{Av}, and hence it admits a Cohen factorization
$Q\xra{\dot{\theta}} Q'\xra{\theta'} \widehat{T_{\mathfrak M}}$ such that $\theta'$ is a deformation by~\ref{para2}. This yields that the following diagram of local homomorphisms is a quasi-deformation.
\begin{equation}\label{eq130610c}
S\xra{\alpha} \widehat{T_{\mathfrak M}}\xla{\theta'} Q'
\end{equation}

The isomorphism $T\otimes_{S} M\cong R'\otimes_{R} M$ of $Q$-modules implies that $\fd_{Q}(T\otimes_{S} M)<\infty$.
Since $\widehat{T_{\mathfrak M}}\otimes_{T}(T\otimes_{S} M)\cong \widehat{T_{\mathfrak M}}\otimes_{S} M$ as $Q$-modules and $\widehat{T_{\mathfrak M}}$ is a flat $T$-module, we have $\fd_Q(\widehat{T_{\mathfrak M}}\otimes_{S} M)<\infty$. Also, $\widehat{T_{\mathfrak M}}\otimes_{S} M$ is finitely generated over $\widehat{T_{\mathfrak M}}$. Now using~\cite[Lemma(3.2)]{AFH1} we get $\pdim_{Q'}(\widehat{T_{\mathfrak M}}\otimes_{S} M)<\infty$, and hence $\cidim_S M<\infty$.

(ii) We continue with the notations of the proof of part (i) and assume that $\cidd_R M<\infty$ and hence diagram~\eqref{eq250803d} is weakly c.i., so $\delta$ is c.i.\ at $\m'$, and hence $\DD_2(R'|R,k')=0$, by~\cite[Proposition (1.8)]{Av}. By an argument similar to the first part, we get $\DD_2(\widehat{T_{\mathfrak M}}|S,E)=0$. Therefore~\cite[Proposition (1.8)]{Av} implies that $\alpha$ is c.i. at $\mathfrak M\widehat{T_{\mathfrak M}}$ and hence diagram~\eqref{eq130610c} is weakly c.i. As in the end of the proof of the previous part, we get $\cid_S M<\infty$.

(iii) Again continue with the notations of the proof of part (i) and assume that $\cfdd_R M<\infty$ and hence diagram~\eqref{eq250803d}
is an upper quasi-deformation, so $R'/\m R'$ is regular. We have the isomorphism $T/(\n\otimes_{R} R')\cong \ell\otimes_{k}R'/\m R'$ of Noetherian rings.
Since $\ell$ is separable over $k$, we deduce from~\cite[Theorem (26.9)]{Mat} that $\ell$ is 0-smooth over $k$. Now by~\cite[Theorem 1]{M}, the closed fiber of $\alpha$ is regular, so diagram~\eqref{eq130610c} is an upper quasi-deformation. As in the proof of part~(i), we get $\ciddim_S M<\infty$.
\end{proof}

Next, we obtain the conclusions of Theorem~\ref{thm3.1} when $\vf$ is  surjective, then for the case where $\vf$ has finite flat dimension.

\begin{prop}
    \label{prop250803a}
Let $M$ be a finitely generated $S$-module. If $\varphi$ is surjective and c.i.\ at $\n$, then the following hold.
\begin{itemize}
\item[(i)]
If $\cfd_R M<\infty$, then $\cidim_S M<\infty$.
\item[(ii)]
If $\cidd_R M<\infty$, then $\cid_S M<\infty$.
\item[(iii)]
If $\cfdd_R M<\infty$, then $\ciddim_S M<\infty$.
\end{itemize}
\end{prop}

\begin{proof}
Fix a quasi-deformation
\begin{equation}
R\xra{\delta} (R',\m',k')\xla{\gamma} Q \label{eq250803a}
\end{equation}
and consider the following diagram of local ring homomorphisms
$$\xymatrix{
R\ar[r]^-\delta\ar[d]_\vf
&R'\ar[d]_{\vf'}
&Q\ar[l]_-\gamma\ar[ld]^-{\gamma'=\vf'\gamma} \\
S\ar[r]_-{\delta'}&S'=\Otimes[R]{S}{R'}
}$$
where $\delta'$ and $\vf'$ are the natural maps.
Then $\delta'$ is flat by base change with the same closed fibre as $\delta$ since $\vf$ is surjective. Our assumptions on $\vf$ imply that it is surjective with kernal generated by an $R$-regular sequence, so $\vf'$ is surjective with kernal generated by an $R'$-regular sequence by flat base change. Thus, the diagram \begin{equation}
S\xra{\delta'}(S',\n')\xla{\gamma'}Q\label{eq250803b}
\end{equation} is a quasi-deformation. Furthermore, since $\delta$ and $\delta'$ have the same closed fibre, if~\eqref{eq250803a} is an
upper quasi-deformation or a complete intersection quasi-deformation, then so is~\eqref{eq250803b}.

(i) Assume that $\cfd_R M<\infty$, so there is a quasi-deformation~\eqref{eq250803a} such that $\fd_Q(\Otimes[R]{R'}{M})<\infty$. It follows that the quasi-deformation~\eqref{eq250803b} satisfies
\begin{align*}
\Otimes[S]{S'}{M}
\cong \Otimes[S]{(\Otimes[R]{S}{R'})}{M}
\cong\Otimes[R]{R'}{M}.
\end{align*}
The fact that $M$ is finitely generated over $S$ implies that $\Otimes[S]{S'}{M}$ is finitely generated over $S'$ and $Q$, so
\begin{align*}
\pd_Q(\Otimes[S]{S'}{M})
=\fd_Q(\Otimes[S]{S'}{M})
=\fd_Q(\Otimes[R]{R'}{M})<\infty
\end{align*}
which implies $\cidim_S M<\infty$, as desired.

(ii)
and
(iii)
follow similarly since
if~\eqref{eq250803a} is an
upper quasi-deformation or a complete intersection quasi-deformation, then so is~\eqref{eq250803b}.
\end{proof}

\begin{cor}
    \label{cor250803a}
Let $M$ be a finitely generated $S$-module. If $\varphi$ is essentially of finite type and c.i.\ at $\n$, then the following hold.
\begin{itemize}
\item[(i)]
If $\cfd_R M<\infty$, then $\cidim_S M<\infty$.
\item[(ii)]
If $\cidd_R M<\infty$, then $\cid_S M<\infty$.
\item[(iii)]
Let $\ell$ be separable over $k$. If $\cfdd_R M<\infty$, then $\ciddim_S M<\infty$.
\end{itemize}
\end{cor}

\begin{proof}
(i) Assume that $\cfd_R M<\infty$.
Since $\varphi$ is essentially of finite type, it admits a regular factorization $R\xra{\dot\vf} R'=R[X]_{\mathfrak M}\xra{\vf'} S$.
The fact that $\dot\vf$ is weakly regular implies
that it is c.i.\ at $\n$, so Theorem~\ref{thm3.1}(i) implies that $\cidim_{R'} M<\infty$.
The fact that
$\vf$ is c.i.\ at $\n$ implies that $\vf'$ is also c.i.\ at $\n$, so Proposition~\ref{prop250803a}(i) implies that
$\cidim_S M<\infty$, as desired.

(ii) and (iii) follow similarly using the appropriate items in Theorem~\ref{thm3.1} and Proposition~\ref{prop250803a}. For (iii), note that $R'$ and $S$ have the same residue field $\ell$, so  residue-field separability for $R\to S$ implies the same for $R\to R'$.
\end{proof}

Our next theorem extends a result of Majadas Soto~\cite[Proposition 12]{S}.
It is stated in terms of
Avramov and Iyengar's~\cite{AI2}
\emph{Andr\'{e}-Quillen dimension} of the $R$-algebra $S$
\begin{equation}
\aq_RS:=\sup\{n\in \bbn\mid \DD_n(S| R,-)\neq 0\}\label{eq250805a}
\end{equation}
in particular, we have $\aq_{R} S=-\infty$ if and only if
$\DD_n(S|R,-)=0$, for all integers $n$.

\begin{thm}\label{thm3.2}
Assume that $\varphi$ is essentially of finite type and $\cfd_{R} S<\infty$. Then $\aq_{R} S\leq 2$ if and only if $\HH_1(K(\vf))$ is a free $\widehat{S}$-module.
\end{thm}

\begin{proof}
The forward implication holds by~\cite[Lemma (2.12)]{Sh}.
For the converse, assume that $\HH_1(K(\vf))$ is a free $\widehat{S}$-module.
Since $\varphi$ is essentially of finite type, it admits a regular factorization $R\rightarrow R'=R[X]_{\mathfrak M}\rightarrow S$. Hence,
$R\xra{\dot{\vf}} \widehat{R'}\xra{\vf'} \widehat{S}$ is a Cohen factorization for $\grave{\vf}$.
Since $\cfd_{R} S<\infty$, from Theorem~\ref{thm3.1}(i), we get $\cdim_{R'} S<\infty$, and hence $\cdim_{\widehat{R'}} \widehat{S}<\infty$, by~\cite[Proposition (1.13)(2)]{AGP}.

Let $E$ be the Koszul complex associated to an arbitrary generating set for $\Ker{\vf'}$.
By~\cite[Lemma (2.11)]{Sh} the freeness of $\HH_1(K(\vf))$ as an $\widehat{S}$-module implies
that $H_1(E)$ is also free over $\widehat{S}$, and hence $D_n(\widehat{S}|\widehat{R'},\ell)=0$ for $n\geq 3$ by ~\cite[Proposition 12]{S}.
Now from~\cite[Lemma (1.7)]{Av} it follows that $D_n(S|R,\ell)=0$ for $n\geq 3$, and since $\vf$ is essentially of finite type, from~\cite[Lemma (8.7)]{I} we get $\aq_{R} S\leq 2$.
\end{proof}

Next we recover a result of Majadas Soto~\cite[Proposition 12]{S} and Avramov,
Henriques and \c{S}ega~\cite[(2.5)(3)]{AHS}.

\begin{cor}\label{cor3.3}
Assume that $\varphi$ is surjective, and let $E$ be the Koszul complex associated to a set of generators for $J:=\Ker \varphi$.
If $\cdim_{R} S<\infty$ and $H_1(E)$ is a free $S$-module, then $\aq_{R} S\leq 2$.
\end{cor}

\begin{proof}
Since a generating set over $R$ for $J$ induces a generating set over $\widehat{R}$ for  $\widehat{J}:=\Ker\widehat{\varphi}$, thus~\cite[Proposition (1.6.7)(b)]{BH} and~\cite[Lemma (2.11)]{Sh}
imply that $\HH_1(K(\varphi)$ is a free $\widehat{S}$-module, and hence $\aq_{R} S\leq 2$ by Theorem~\ref{thm3.2}.
\end{proof}

As an application of Theorems~\ref{thm3.1}(ii) and~\ref{thm3.2}, we have our next corollary.

\begin{cor}\label{cor3.4}
Assume that $\sigma$ is essentially of finite type and $\sigma\varphi$ is c.i.\ at $\fr$. Consider the following conditions:

\begin{itemize}
\item[(i)]
$\cfdd_S T<\infty$.
\item[(ii)]
$\cidd_S T<\infty$.
\item[(iii)]
$\cfd_{S} T<\infty$, and $\HH_1(K(\sigma))$ is a free $\widehat{T}$-module.
\item[(iv)]
$\aq_S T<\infty$.
\item[(v)]
$\aq_{S}T\leq 2$.
\end{itemize}
Then, all of the implications $(i)\implies(ii)\implies(iii)\implies(iv)\implies(v)$ hold. If $S$ and $T$ are complete, then the conditions (i)--(v) are equivalent.
\end{cor}

\begin{proof}
(i)$\implies$(ii) Follows from~\ref{fact130625a}.

(ii)$\implies$(iii).
Assume that $\cidd_S T<\infty$.
Then $\cfd_{S} T<\infty$ by~\ref{fact130625a}. Since $\sigma$ is essentially of finite type,
it has a regular factorization $S\xra{\dot{\sigma}} S'=S[X]_{\mathfrak M}\xra{\sigma'} T$. Since $\dot{\sigma}$ is essentially of finite type and flat with regular closed fiber
and $\cidd_S T<\infty$, it follows from Theorem~\ref{thm3.1}(ii) that $\cid_{S'} T<\infty$.

From~\ref{defn130610a}, we have
$\HH_1(K(\sigma))\cong\HH_1(K(\widehat{\sigma'}))\cong\HH_1(K(\sigma'))$, as $\widehat{T}$-modules. Since $\sigma'(\dot{\sigma}\varphi)=\sigma\varphi$
is c.i.\ at $\fr$ and $\sigma'$ is surjective,
we conclude that $\HH_1(K(\sigma))$ is a free $\widehat{T}$-module by~\cite[Theorem A (iv)$\implies$(v)]{Sh}.

(iii)$\implies$(iv) Follows from Theorem~\ref{thm3.2}.

(iv)$\implies$(v).
Assume that $\aq_S T<\infty$.

Since $\sigma\varphi$ is \emph{c.i.} at $\fr$, from the Jacobi-Zariski exact sequence arising
from diagram $R\xra{\varphi}S\xra{\sigma}T$;~\cite[Proposition (1.8)]{Av} and~\cite[(3.20)]{An2} we have
$\DD_n(S|R,v)\cong\DD_n(S|R,\ell)\otimes_{\ell}v=0$ for $n\gg 0$; and hence $\DD_n(S|R,\ell)=0$ for $n\gg 0$.
Since $\sigma$ is essentially of finite type, using~\cite[Corollary (4.9), Theorem (6.4)(i)$\implies$(v)]{AI2},
~\cite[Proposition (1.8)]{Av},~\cite[(3.20)]{An2},
the Jacobi-Zariski exact sequence arising from the above diagram and~\cite[Lemma (8.7)]{I}, the conclusion $\aq_{S}T\leq 2$ holds.


(v)$\implies$(i). Assume that $S$ and $T$ are complete and $\aq_{S}T\leq 2$. As in the proof of the previous implication, we have
$\DD_2(S|R,\ell)=0$, and hence $\varphi$ is c.i.\ at $\n$ by~\cite[Proposition (1.8)]{Av}.
Using Construction~\ref{fact130607c}, we deduce from~\ref{para2} and~\cite[Lemma (3.2)]{AFH1} that diagram $S\xra{\id}S\xla{\vf'} R'$
is an upper quasi-deformation such that $\fd_{R'}(T\otimes_{S} S)=\fd_{R'}T<\infty$. Therefore $\cfdd_S T<\infty$.
\end{proof}

Our final result recovers~\cite[Theorem A]{Sh} as a consequence of Corollary~\ref{cor3.4}.

\begin{cor}\label{cor3.5}
Assume that $\sigma$ is surjective, and $\sigma\varphi$ is complete intersection at $\fr$. Then the following conditions are equivalent.

\begin{itemize}
\item[(i)]
$\aq_S T<\infty$.
\item[(ii)]
$\aq_S T\leq 2$.
\item[(iii)]
$\ciddim_{S} T<\infty$.
\item[(iv)]
$\cid_S T<\infty$.
\item[(v)]
$\cidim_S T<\infty$, and $\HH_1(K(\sigma))$ is a free $\widehat{T}$-module.
\end{itemize}
\end{cor}

\begin{proof}
By Corollary~\ref{cor3.4},
we need only prove the implication (ii)$\implies$(iii).
Assume that $\aq_S T\leq 2$.
From~\ref{para1},~\ref{para2},~\ref{para3}, and~\cite[(4.57)]{An2} we see that $\widehat{\sigma}\grave{\varphi}$ is c.i.\ at $\fr\widehat{T}$ and
$\aq_{\widehat{S}}\widehat{T}\leq 2$. By Corollary~\ref{cor3.4}(v)$\implies$(i) and~\cite[Fact (2.4)(b)]{Sh} the desired condition
$\ciddim_{S} T<\infty$ holds.
\end{proof}

\providecommand{\bysame}{\leavevmode\hbox to3em{\hrulefill}\thinspace}
\providecommand{\MR}{\relax\ifhmode\unskip\space\fi MR }
\providecommand{\MRhref}[2]{%
  \href{http://www.ams.org/mathscinet-getitem?mr=#1}{#2}
}
\providecommand{\href}[2]{#2}


\begin{thebibliography}{10}


\bibitem{An2}
M. Andr\'{e}, \emph{Homologie des alg\`ebres commutatives}, Springer-Verlag,
Berlin, 1974, Die Grundlehren der mathematischen Wissenschaften, Band 206.

\bibitem{Av}
L.~L. Avramov, {\em Locally complete intersection homomorphisms and a conjecture of Quillen
on the vanishing of cotangent homology}, Ann. of Math. {\bf 150} (1999), 455--487.

\bibitem{AFH1}L.~L. Avramov, H.-B.\ Foxby, and B.\ Herzog, \emph{Structure of local
homomorphisms}, J. Algebra \textbf{164} (1994), 124--145.

\bibitem{AGP}
L.~L. Avramov, V.~N. Gasharov, and I.~V. Peeva, \emph{Complete intersection dimension}, Inst. Hautes \'Etudes Sci. Publ. Math. (1997), no.~86, 67--114
(1998).

\bibitem{AHS}
L.~L. Avramov, I.~B.~Henriques, and L.~M.~Sega, {\em Quasi-complete intersection homomorphisms},
Pure and Applied Mathematics Quarterly {\bf 9}  (2013), no.~4, 1--31.

\bibitem{AI2}
L.~L.~Avramov and S.~B.~Iyengar, {\em Andr\'{e}-Quillen homology of algebra retracts}, Ann. Sci. \'{E}cole Norm. Sup. (4) {\bf 36} (2003),
431--462.

\bibitem{BH}
W.~Bruns and J.~Herzog, \emph{Cohen-{M}acaulay rings}, revised ed., Studies in
Advanced Mathematics, vol.~39, University Press, Cambridge, 1998.

\bibitem{I}
S.~B.~Iyengar, {\em Andr\'{e}-Quillen homology of commutative
algebras}, Interactions between homotopy theory and algebra,
Contemp. Math. {\bf 436}, Amer. Math. Soc., Providence, RI, 2007.

\bibitem{IW}
S.~B.~Iyengar and K.~A.~Sather-Wagstaff, {\em G-dimension over local
homomorphisms. Applications to the Frobenius endomorphism},
Illinois J. Math. {\bf 48} (2004), no.~1, 241--272.

\bibitem{M}
J. Majadas, {\em On tensor products of complete intersections}, Bull. Lond. Math. Soc. {\bf 45} (2013), no.~6, 1281-1284.

\bibitem{M1}
J. Majadas, {\em A descent theorem for formal smoothness}, Nagoya, Math. Journal, {\bf 229} (2018), 113-140.

\bibitem{Mat}
H. Matsumura, {\em Commutative ring theory}, Stud. Adv. Math. {\bf 8}, Univ. Press, Cambridge, 1986.

\bibitem{Q2}
D.~Quillen, {\em Homology of commutative rings}, mimeographed notes, MIT 1968.

\bibitem{SSW1}
K.~A.~Sather-Wagstaff, \emph{Complete intersection dimensions and {F}oxby
classes}, J. Pure Appl. Algebra \textbf{212} (2008), no.~12, 2594--2611.

\bibitem{Sh}
T. Sharif, {\em Andr\'{e}-Quillen homology and complete intersection dimensions}, Kodai. Math. Journal. {\bf 47} (2024), no.~2, 215--230.

\bibitem{S}
J. J. M. Soto, {\em Finite complete intersection dimension and vanishing of
Andr\'{e}-Quillen homology}, J. Pure Appl. Algebra {\bf
146} (2000), 197--207.

\bibitem{RT}
R. Takahashi, {\em Upper complete intersection dimension relative to a local homomorphism}, Tokyo J. Math. {\bf
27} (2004), 209--219.

\end{thebibliography}
\end{document}